\documentclass[a4paper,12pt,leqno]{amsart}

\usepackage{latexsym}
\usepackage[all]{xy}

\usepackage{amsmath, amsfonts, amssymb,amscd,graphics,graphicx,color,eucal,setspace,xypic} 
\usepackage{comment}
\usepackage{mathtools}

\usepackage{graphicx}
\usepackage{stmaryrd}
\usepackage{bigdelim, multirow}%for underbrace
\usepackage[all]{xy}%for commutative diagram
\usepackage{tikz-cd}

\definecolor{gray}{gray}{0.7}
\definecolor{Gray}{gray}{0.3}

\usepackage{amscd}

\numberwithin{equation}{section}

\theoremstyle{break}
 \newtheorem{theorem}{Theorem}[section]
 \newtheorem{proposition}[theorem]{Proposition}
 \newtheorem{corollary}[theorem]{Corollary}
 \newtheorem{lemma}[theorem]{Lemma}

 \theoremstyle{definition}
 
 \newtheorem{remark}[theorem]{Remark}
 \newtheorem{example}[theorem]{Example}

\allowdisplaybreaks[4]

\def\C{\mathbb C}
\def\R{\mathbb R}
\def\Q{\mathbb Q}
\def\Z{\mathbb Z}

\def\g{\mathfrak{g}}
\def\b{\mathfrak{b}}
\def\t{\mathfrak{t}}
\def\u{\mathfrak{u}}
\def\p{\mathfrak{p}}
\def\a{\mathfrak{a}}
\def\Levi{L}

\def\SS{\mathcal{S}}
\def\RR{\mathcal{R}}

\def\xx{\mathsf{x}}
\def\yy{\mathsf{y}}
\def\ss{\mathsf{s}}
\def\nn{\mathsf{n}}
\def\vv{\mathsf{v}}
\def\ee{\mathsf{e}}
\def\PP{\mathsf{P}}

\DeclareMathOperator{\reg}{r}
\DeclareMathOperator{\regsemi}{rs}
\DeclareMathOperator{\pt}{pt}
\DeclareMathOperator{\rank}{rank}
\DeclareMathOperator{\spanned}{span}
\DeclareMathOperator{\cone}{cone}
\DeclareMathOperator{\conv}{conv}
\DeclareMathOperator{\Map}{Map}

\DeclareMathOperator{\Sym}{Sym}

\DeclareMathOperator{\Poin}{Poin}
\DeclareMathOperator{\Ad}{Ad}

\DeclareMathOperator{\SL}{SL}
\DeclareMathOperator{\SU}{SU}

\DeclareMathOperator{\Hess}{Hess}
\DeclareMathOperator{\Pet}{Pet}
\DeclareMathOperator{\Flag}{Fl}
\begin{document}
\title[Partial Hessenberg varieties]{Notes on the cohomology of partial Hessenberg varieties}
\author[T. Horiguchi]{Tatsuya Horiguchi}
\address{National Institute of Technology, Akashi College, 679-3, Nishioka, Uozumi-cho, Akashi, Hyogo, 674-8501, Japan}
\email{tatsuya.horiguchi0103@gmail.com}

\author[M. Masuda]{Mikiya Masuda}
\address{Osaka Central Advanced Mathematical Institute, Osaka Metropolitan University,
Sugimoto, Sumiyoshi-ku, Osaka, 558-8585, Japan}
\email{mikiyamsd@gmail.com}

\author[T. Sato]{Takashi Sato}
\address{Osaka Central Advanced Mathematical Institute, Osaka Metropolitan University,
Sugimoto, Sumiyoshi-ku, Osaka, 558-8585, Japan}
\email{00tkshst00@gmail.com}

\author[H. Zeng]{Haozhi Zeng}
\address{School of Mathematics and Statistics, Huazhong University of Science and Technology, Wuhan, 430074, P.R. China}
\email{zenghaozhi@icloud.com}



\begin{abstract}
Hessenberg varieties are a family of subvarieties of full flag varieties. This family contains well-known varieties such as Springer fibers, Peterson varieties, and permutohedral varieties. It was introduced by De\! Mari-Procesi-Shayman in 1992 and has been actively studied in this decade. In particular, unexpected relations to hyperplane arrangements and the Stanley-Stembridge conjecture in graph theory have been discovered. Hessenberg varieties can be defined in partial flag varieties. In this paper, we study their cohomology by relating them to the cohomology of Hessenberg varieties in the full flag varieties. 
\end{abstract}

\maketitle

\setcounter{tocdepth}{1}

\tableofcontents

%%%%%%%%%%%%%%%%%%%%%%%%%%%%%%%%%%
\section{Introduction}
\label{sect:Intro}

Let $G$ be a simply connected semisimple linear algebraic group over $\C$ of rank $n$. We choose and fix a Borel subgroup $B$ of $G$ and a maximal torus $T$ in $B$, so the set of simple roots $\Delta$ is determined. A $B$-submodule $H$ of the Lie algebra $\g$ of $G$ (under the adjoint action) is called a Hessenberg space if it contains the Lie algebra $\b$ of $B$. 
Then, for each $\xx\in \g$, the Hessenberg variety $\Hess(\xx,H)$ is defined in \cite{dMPS} by 
\[
\Hess(\xx,H):=\{gB\in G/B\mid \Ad(g^{-1})(\xx)\in H\}. 
\] 
The family of Hessenberg varieties contains Springer fibers in geometric representation theory, Peterson varieties related to the quantum cohomology of the flag variety $G/B$, and varieties related to QR algorithm to find eigenvalues of matrices. Here the family of the varieties related to QR algorithm  contains permutohedral varieties which are toric varieties.  
Recently, unexpected relations to hyperplane arrangements (\cite{AHMMS}) and the Stanley-Stembridge conjecture in graph theory (\cite{BroCho})\footnote{The Stanley-Stembridge conjecture was recently solved by Hikita \cite{Hik} through purely combinatorial argument.} have been discovered. 

One can define what is called a \emph{partial} Hessenberg variety in a partial flag variety $G/P$ where $P$ is a parabolic subgroup of $G$ containing $B$. In this case we require that $H$ be a $P$-submodule of $\g$ containing the Lie algebra $\p$ of $P$. We call such $H$ a $\p$-Hessenberg space. A $\b$-Hessenberg space is a Hessenberg space mentioned above. For a $\p$-Hessenberg space $H$, we define 
\[
\Hess_{\Theta}(\xx,H):=\{ gP\in G/P\mid \Ad(g^{-1})(\xx)\in H\},
\]
where $\Theta$ denotes a subset of the simple roots $\Delta$ corresponding to $P$. When $P=B$, $\Theta$ is the empty set. 

In this paper, we study the cohomology of $\Hess_{\Theta}(\xx,H)$ by relating it to the cohomology of $\Hess(\xx,H)$. This is done in \cite{Hor24} when $\xx$ is regular nilpotent.   
A key observation is that the projection $G/B\to G/P$ induces a fibration 
\begin{equation} \label{eq:intro_fibration}
P/B\to \Hess(\xx,H)\xrightarrow{\pi_H} \Hess_{\Theta}(\xx,H), 
\end{equation}
and the group $W_\Theta$ associated to $\Theta$ acts on these spaces, making this fibration $W_\Theta$-equivariant, where the action of $W_\Theta$ on $\Hess_{\Theta}(\xx,H)$ is trivial. In fact, 
the $W_\Theta$-action on $\Hess(\xx,H)$ is the restriction of the natural right action of the Weyl group $W$ on $G/B$. It is not algebraic since it is defined through the natural identification of $G/B$ with $K/(K\cap T)$ where $K$ is a maximal compact Lie subgroup of $G$. The induced $W_\Theta$-action on $H^*(\Hess(\xx,H))$ is called the \emph{star action}. Throughout this paper, all cohomology rings will be taken with $\Q$-coefficients unless otherwise stated. Our main result is the following. 

\begin{theorem}[Theorem~\ref{theorem:main}] \label{theo:intro_main}
For any $\xx\in \g$ and any $\p$-Hessenberg space $H$, the following hold: 
\begin{enumerate}
\item $\pi_H^*\colon H^*(\Hess_{\Theta}(\xx,H))\to H^*(\Hess(\xx,H))$ is injective.
\item $H^*(\Hess(\xx,H))\cong H^*(P/B)\otimes_\Q H^*(\Hess_{\Theta}(\xx,H))$ as graded $W_\Theta$-modules.
\item $H^*(\Hess_\Theta(\xx,H))\cong H^*(\Hess(\xx,H))^{W_\Theta({\rm star})}$ as graded $\Q$-algebras. 
\end{enumerate}
Here $H^*(\Hess(\xx,H))^{W_\Theta({\rm star})}$ denotes the invariants in $H^*(\Hess(\xx,H))$ under the star action of $W_\Theta$. 
\end{theorem}

An element $\ss \in \g$ is \emph{regular semisimple} if its centralizer in $G$ is a maximal torus of $G$.
By the adjoint action on $\g$, we may assume that the maximal torus is the $T$ we chose at the beginning. Then $\Hess(\ss,H)$ is invariant under the natural left $T$-action on $G/B$. It is known that $\Hess(\ss,H)$ with this $T$-action is a GKM manifold and $H^*(\Hess(\ss,H))$ becomes a $W$-module through the GKM theory in \cite{GKM}. The $W$-action on $H^*(\Hess(\ss,H))$ was introduced by Tymoczko in \cite{Tym08} and called the \emph{dot action}. The star action of $W_\Theta$ and the dot action of $W$ on $H^*(\Hess(\ss,H))$ commute with each other. 

\begin{proposition}[Proposition~\ref{proposition:dotaction starinvariants}] \label{prop:intro_prop}
When $H$ is a $\p$-Hessenberg space, we have an isomorphism 
\[
H^*(\Hess(\ss,H))\cong H^*(\Hess(\ss,H))^{W_\Theta({\rm star})}\otimes_\Q H^*(P/B)\quad \text{as graded $W$-modules}
\]
with respect to the dot action, where the $W$-action on $H^*(P/B)$ is trivial. 
\end{proposition}

An element $\yy\in\g$ is \emph{regular} if its centralizer in $G$ has dimension $n$. We may assume that $\yy$ is associated to a subset $\Xi$ of $\Delta$ similarly to the regular semisimple case, and we have a subgroup $W_\Xi$ of $W$ associated to $\Xi$. Combining the observation above with the argument in \cite{BaCr24}, we obtain the following. 

\begin{theorem}[Theorem~\ref{theorem:main2}] \label{theo:intro_main2}
When $H$ is a $\p$-Hessenberg space, we have an isomorphism 
\[
H^*(\Hess_\Theta(\yy,H))\cong H^*(\Hess_\Theta(\ss,H))^{W_\Xi({\rm dot})}\quad\text{as graded $\Q$-algebras},
\]
where the right hand side denotes the invariants in $H^*(\Hess_\Theta(\ss,H))$ under the dot action of $W_\Xi$.
\end{theorem}

The paper is organized as follows. We observe the fibration \eqref{eq:intro_fibration} in Section~\ref{sect:partial Hessenberg varieties} and prove Theorem~\ref{theo:intro_main} in Section~\ref{sect:cohomology} by applying the Leray-Hirsch theorem. For that, we show the existence of a $W_\Theta$-equivariant section to the restriction map $H^*(\Hess(\xx,H))\to H^*(P/B)$. In Section~\ref{sect:equivariant cohomology of regular semisimple} we treat the regular semisimple case. We review the dot action and prove Proposition~\ref{prop:intro_prop}. In Section~\ref{sect:cohomology of regular} we treat the regular case and prove Theorem~\ref{theo:intro_main2}. We also remark that $H^*(\Hess(\yy,H))$ for a regular element $\yy\in \g$ is a Poincar\'e duality algebra, has Hard Lefschetz property, and satisfies Hodge-Riemann relations (Corollary~\ref{coro:regular}). A regular semisimple Hessenberg variety of double lollipop type (\cite{MasSat}) has a structure of a fiber bundle where the fiber is a product of two full flag varieties and the base is a \emph{toric} regular semisimple Hessenberg variety in a partial flag variety.  In Section 6, we give a formula of its Poincar\'e polynomial.

%%%%%%%%%%%%%%%%%%%%%%%%%%%%%%%%%%

\bigskip
\noindent \textbf{Acknowledgements.} 
The first author is supported in part by JSPS KAKENHI Grant-in-Aid for Early-Career Scientists: 23K12981.
The second author is supported in part by JSPS KAKENHI Grant-in-Aid for Scientific Research (C): 25K07007.
The third author is supported in part by JSPS KAKENHI Grant-in-Aid for Scientific Research (C): 25K00205.
The fourth author is supported in part by NSFC: 11901218.

%%%%%%%%%%%%%%%%%%%%%%%%%%%%%%%%%%
\section{Partial Hessenberg varieties} \label{sect:partial Hessenberg varieties}
%%%%%%%%%%%%%%%%%%%%%%%%%%%%%%%%%%

Let $G$ be a simply connected semisimple linear algebraic group over $\C$ of rank $n$.
Fix a Borel subgroup $B$ of $G$ and let $T$ be a maximal torus of $G$ in $B$. 
We write $\t \subset \b \subset \g$ for the Lie algebras of $T \subset B \subset G$. 
Let $\Phi$ be the root system of $\t$ in $\g$. 
We denote by $\Phi^+$ the set of positive roots in $\Phi$, i.e.\ the set of roots of $\t$ in the Lie algebra $\u$ of the unipotent radical of $B$.
We also denote by $\Phi^-$ the set of negative roots, i.e.\ $\Phi^- = \{-\alpha \in \Phi \mid \alpha \in \Phi^+ \}$. 
Let $\Delta =\{\alpha_1,\ldots, \alpha_n \}$ be the set of simple roots.
The root space associated to a root $\alpha$ is denoted by $\g_\alpha$ and we fix a basis $E_\alpha$ of $\g_\alpha$.
Let $P$ be a parabolic subgroup of $G$ containing $B$. 
Recall that $P$ is written as a semidirect product $U_\Theta \rtimes \Levi_\Theta$ for some $\Theta \subset \Delta$ where $U_\Theta$ is the unipotent radical of $P$ and $\Levi_\Theta$ is defined as follows. 
Let $\Phi_\Theta^+$ be the set of positive roots which are linear combinations of roots in $\Theta$. 
We denote by $\Phi_\Theta$ the root system associated to $\Theta$, i.e.\ $\Phi_\Theta=\Phi_\Theta^+ \cup \Phi_\Theta^-$ where $\Phi^-_\Theta=\{-\alpha \in \Phi \mid \alpha \in \Phi^+_\Theta \}$.
Then $\Levi_\Theta$ is the subgroup of $G$ generated by $T$ together with the root subgroups $U_\alpha =\{\exp(\lambda E_\alpha) \in G \mid \lambda \in \C \}$ for $\alpha \in \Phi_\Theta$.
Throughout the article, we fix $\Theta \subset \Delta$ and $P$ is assumed to be the parabolic subgroup $U_\Theta \rtimes \Levi_\Theta$ associated to $\Theta$. 
Let $W = N_G(T)/T$ be the Weyl group of $G$ where $N_G(T)$ is the normalizer of $T$ in $G$. 
We denote by $s_\alpha$ the reflection associated to a root $\alpha \in \Phi$. 
In particular, we write $s_i = s_{\alpha_i}$ for the simple reflection. 
Note that $W$ is generated by the simple reflections $s_1, \ldots, s_n$.
Let $W_\Theta$ be the subgroup of $W$ generated by simple reflections $s_i$ associated to $\alpha_i \in \Theta$.
We denote the Lie algebra of $P$ by $\p$. 

A subspace $H \subset \g$ is called a \emph{$\p$-Hessenberg space} if $H$ is $\p$-submodule and $H$ contains $\p$. 
We simply call a $\b$-Hessenberg space a \emph{Hessenberg space}. 
Note that a $\p$-Hessenberg space $H$ can be written as 
\begin{align} \label{eq:pHessenberg space}
H = \b \oplus \bigoplus_{\alpha \in I(H)} \g_{-\alpha} 
\end{align}
for some subset $I(H) \subset \Phi^+$. 

\begin{remark}
There is a notion introduced in \cite[Definition~2.3]{Hor24} that a subset of $\Phi^+$ is a \emph{$\Theta$-ideal}. 
The subset $I(H) \subset \Phi^+$ described in \eqref{eq:pHessenberg space} is a $\Theta$-ideal. 
Moreover, the correspondence sending $H$ to $I(H)$ gives a one-to-one correspondence between the set of $\p$-Hessenberg spaces and the set of $\Theta$-ideals (\cite[Lemma~2.5]{Hor24}).
Note that the action of $W_\Theta$ on $\Phi$ preserves $I(H) \setminus \Phi_\Theta^+$ (\cite[Lemma~6.2]{Hor24}).
\end{remark}

The \emph{partial Hessenberg variety} $\Hess_\Theta(\xx,H)$ associated to $\xx \in \g$ and $\p$-Hessenberg space $H$ is defined to be the following subvariety of the partial flag variety $G/P$:
\begin{align*}
\Hess_\Theta(\xx,H) \coloneqq \{gP \in G/P \mid \Ad(g^{-1})(\xx) \in H \}.
\end{align*}
Note that the natural right action of $P$ on $G$ preserves a subset $G_H(\xx) = \{g \in G \mid \Ad(g^{-1})(\xx) \in H \}$ since $H$ is a $\p$-Hessenberg space, so $\Hess_\Theta(\xx,H)$ is identified with the quotient space $G_H(\xx)/P$. 
If $P=B$, then $\Hess_\Theta(\xx,H)$ is denoted by $\Hess(\xx,H)$ which is also simply called a \emph{Hessenberg variety}.
For $\xx, \xx' \in \g$ with $\xx'=\Ad(g')(\xx)$ for some $g' \in G$, we obtain the isomorphism $\Hess_\Theta(\xx,H) \cong \Hess_\Theta(\xx',H)$ which sends $gP$ to $g'gP$.
Recall that the natural projection $\pi: G/B \to G/P$ is a fiber bundle with fiber $P/B$. 
By restricting the projection $\pi$ to a Hessenberg variety, one can see that the natural projection $\pi_H: G_H(\xx)/B \to G_H(\xx)/P$ is also a fiber bundle with fiber $P/B$.
We record this fact in the following lemma.

\begin{lemma} \label{lemma:fiber bundle}
Let $\xx \in \g$ and $H$ be a $\p$-Hessenberg space. 
Then the natural projection $\pi_H: \Hess(\xx,H) \to \Hess_\Theta(\xx,H)$ is a fiber bundle with fiber $P/B$. 
\end{lemma}

\begin{remark}
A structure of fiber bundles for Hessenberg varieties was discussed in \cite{KiLe}. 
It was also used in \cite{Hor24} to study the cohomology rings of regular nilpotent partial Hessenberg varieties.
\end{remark}

%%%%%%%%%%%%%%%%%%%%%%%%%%%%%%%%%%
\section{Cohomology rings of partial Hessenberg varieties and Hessenberg varieties} \label{sect:cohomology}
%%%%%%%%%%%%%%%%%%%%%%%%%%%%%%%%%%

\subsection{Star action} \label{subsection:star action}
We begin with an action of the Weyl group $W$ on the flag variety $G/B$. 
A key point is an identification of $G/B$ with a quotient space of real Lie groups in order to construct a $W$-action on $G/B$.
Let $K \subset G$ be a maximal compact subgroup such that $T_K \coloneqq K \cap T$ is a maximal (compact) torus in $K$. 
The Weyl group $W$ naturally acts on $K/T_K$, which yields a $W$-action on $G/B$ via the homeomorphism $K/T_K \approx G/B$ induced from the inclusion $K \subset G$. 
In fact, let $N_K(T_K)$ be the normalizer of $T_K$ in $K$. 
Then the right action of the Weyl group $W \cong N_K(T_K)/T_K$ on $K/T_K$ is defined by $(kT_K) \cdot (zT_K) = kzT_K$ for $kT_K \in K/T_K$ and $zT_K \in N_K(T_K)/T_K$.
Note that this $W$-action on $G/B$ does not preserve a complex structure.

\begin{example}
The flag variety $\Flag(\C^n)$ in type $A_{n-1}$ is the set of nested linear subspaces $V_\bullet \coloneqq (V_1\subset V_2 \subset \cdots \subset V_n=\C^n)$ of $\C^n$ such that $\dim_\C V_i = i$ for all $1 \leq i \leq n$.
Let $B$ be the set of upper triangular matrices in the special linear group $G=\SL_n(\C)$. 
To each matrix $g \in \SL_n(\C)$ we assign the flag $V_\bullet = (V_i)_{1 \leq i \leq n}$ so that $V_i$ is spanned by the first $i$ column vectors of $g$. 
This correspondence yields an identification $\Flag(\C^n) \cong \SL_n(\C)/B$. 
By a similar argument, $\Flag(\C^n)$ is also identified with the quotient space $\SU(n)/T_K$ where $T_K$ is the set of diagonal matrices in the special unitary group $K=\SU(n)$. 
Recall that the Weyl group $W$ of type $A_{n-1}$ is the symmetric group $\mathfrak{S}_n$ on $n$ letters $\{1,2,\ldots,n\}$.
To see the right action of $\mathfrak{S}_n$ on $\Flag(\C^n)$, we need an identification $\Flag(\C^n) \cong \SU(n)/T_K$. 
In fact, for a flag $V_\bullet = (V_i)_{1 \leq i \leq n} \in \Flag(\C^n)$, we take a matrix $g \in \SL_n(\C)$ such that $V_i$ is spanned by $g_1, \ldots, g_i$ for all $1 \leq i \leq n$ where $g_1, \ldots, g_n$ denotes the column vectors of $g$. 
If we define $V_\bullet \cdot w$ by the flag determined by $V_\bullet \cdot w = (V'_i)_{1 \leq i \leq n}$ such that $V'_i$ is spanned by $g_{w(1)}, \ldots, g_{w(i)}$ for all $1 \leq i \leq n$, then this definition is \emph{not} well-defined. 
However, if we choose a matrix $u$ in the special unitary group $\SU(n)$ instead of $g \in \SL_n(\C)$ and we define $V_\bullet \cdot w$ by the flag determined by $V_\bullet \cdot w = (V'_i)_{1 \leq i \leq n}$ such that $V'_i$ is spanned by $u_{w(1)}, \ldots, u_{w(i)}$ for all $1 \leq i \leq n$, then this definition is well-defined. 
\end{example}

By \cite[Lemma~5.3]{Hor24}, if $H$ is a $\p$-Hessenberg space, then the $W_\Theta$-action on $G/B$, which is the action by the restriction of $W$ to the subgroup $W_\Theta$, preserves a Hessenberg variety $\Hess(\xx,H)$. 
We denote this action by 
\begin{align} \label{eq:right action}
gB \cdot w \in \Hess(\xx,H) \ \ \ \textrm{for all} \ gB \in \Hess(\xx,H) \ \textrm{and} \ w \in W_\Theta.
\end{align}

\begin{lemma} \label{lemma:piHequivariant}
Let $\xx \in \g$ and $H$ be a $\p$-Hessenberg space. 
Let $\pi_H: \Hess(\xx,H) \to \Hess_\Theta(\xx,H)$ be the natural projection.
Then we have $\pi_H(gB \cdot w) = \pi_H(gB)$ for $gB \in \Hess(\xx,H)$ and $w \in W_\Theta$. 
In particular, the action of $W_\Theta$ on $\Hess_\Theta(\xx,H)$ is trivial so that $\pi_H$ is a $W_\Theta$-equivariant map. 
\end{lemma}

\begin{proof}
Let $gB \in \Hess(\xx,H)$ and $w \in W_\Theta$.
We write $gB = kT_K$ under the identification $G/B \approx K/T_K$.
We also denote by $w=zT_K$ where we regard $W_\Theta$ as a subgroup of $N_K(T_K)/T_K$. 
The well-known fact that $P = B W_\Theta B$ yields $z \in P$.
Since $gB \cdot w = (kT_K) \cdot zT_K =kzT_K=kzB$ by definition, we have $\pi_H(gB \cdot w)=kzP=kP=\pi_H(gB)$, as desired.
\end{proof}

For $\xx \in \g$ and a $\p$-Hessenberg space $H \subset \g$, the $W_\Theta$-action on the Hessenberg variety $\Hess(\xx,H)$ yields the $W_\Theta$-action on its cohomology $H^*(\Hess(\xx,H))$. 
In this paper, we call this action the \emph{star action} of $W_\Theta$ on $H^*(\Hess(\xx,H))$, which is denoted by $w \ast \alpha \in H^*(\Hess(\xx,H))$ for $w \in W_\Theta$ and $\alpha \in H^*(\Hess(\xx,H))$.
In particular, if $\Theta=\Delta$ and $\Hess(\xx,H) = G/B$ (for example $\xx=0$ or $H =\g$), then we obtain the star action of $W$ on $H^*(G/B)$. 
We will discuss the star action of $W$ on $H^*(G/B)$ in the next section.

\subsection{Borel's presentation}

Let $\t^*_{\Z}$ be a lattice determined by the character group of $T$ through the differential at the identity element of $T$. 
We may identify $\t^*_{\Z}$ with the character group of $T$ below.
Set $\t^*_{\Q}=\t^*_{\Z} \otimes_{\Z} \Q$ and we denote its symmetric algebra by $\RR = \Sym \t^*_{\Q}$. 
The $W$-action on $\t^*_{\Q}$ naturally extends to a $W$-action on $\RR$. 
In what follows, $(\RR^{W_\Theta}_+)$ denotes the ideal generated by the $W_\Theta$-invariants in $\RR$ with zero constant term. 

For a $B$-module $V$, one can construct a vector bundle $G \times_B V$ over the flag variety $G/B$. 
Here, $G \times_B V$ is the quotient of the direct product $G \times V$ by the left $B$-action given by $b \cdot (g,\vv) = (gb^{-1}, b \cdot \vv)$ for $b \in B, g \in G$, and $\vv \in V$.
We write $[g,\vv]$ for a coset represented by $(g,\vv)$.
We regard the vector bundle $G \times_B V$ as a $K$-vector bundle over $K/T_K$ through the homeomorphism $K/T_K \approx G/B$. 
Then we obtain a homeomorphism 
\begin{align} \label{eq:vector bundle homeo}
K \times_{T_K} V \approx G \times_B V.
\end{align}
Here, any element $(k,\vv) \in K \times V$ determines two cosets in $K \times_{T_K} V$ and $G \times_B V$, respectively. 
Then the homeomorphism in \eqref{eq:vector bundle homeo} is given by an identification of these two cosets (cf. \cite[Proposition~3.2]{Bre72}).

We now define the complex line bundle $L_\chi$ associated to $\chi \in \t^*_\Z$ over $G/B$ as follows. 
Recall that $\chi: T \to \C^*$ extends to a homomorphism $\tilde\chi: B \to \C^*$ since $B = U \rtimes T$. 
Let $\C_{\tilde\chi}$ be the $1$-dimensional $B$-module via $\tilde\chi$ and define $L_\chi = G \times_B \C_{\tilde\chi}$.
An assignment of the first Chern class $c_1(L_\chi)$ of the line bundle $L_\chi$ to each $\chi \in \t^*_\Z$ gives a homomorphism $\RR \rightarrow H^*(G/B)$, which doubles the grading on $\RR$. 
Borel proved in \cite{Bor53} that this homomorphism is surjective and its kernel is $(\RR^{W}_+)$.
In other words, we obtain an isomorphism of graded $\Q$-algebras 
\begin{align} \label{eq:Borel}
H^*(G/B) \cong \RR/(\RR^{W}_+).
\end{align} 
The right action of $w \in W$ on $K/T_K$ gives a homeomorphism $r_w: K/T_K \to K/T_K$.
Let $\C_{\chi}$ be the $1$-dimensional $T$-module via $\chi$. 
Then $L_\chi$ is identified with $K \times_{T_K} \C_\chi$ by \eqref{eq:vector bundle homeo}. 
Since the pullback bundle $r_w^*(K \times_{T_K} \C_\chi)$ is isomorphic to $K \times_{T_K} \C_{w(\chi)}$, we obtain the equality $w \ast c_1(L_\chi) = c_1(L_{w(\chi)})$ in $H^*(G/B)$.
Thus, the correspondence in \eqref{eq:Borel} is also an isomorphism as $W$-modules where the $W$-action on $H^*(G/B)$ is the star action.

Consider the fiber bundle $P/B \xhookrightarrow{\iota} G/B \xrightarrow{\pi} G/P$ and the induced homomorphisms:
\begin{align*}
H^*(G/P) \xrightarrow{\pi^*} H^*(G/B) \xrightarrow{\iota^*} H^*(P/B). 
\end{align*}
It is well-known that the homomorphism $\pi^*$ is injective and its image coincides with $H^*(G/B)^{W_\Theta}$ by \cite[Corollary~5.4 and Theorem~5.5]{BGG}. 
Recall that $P/B$ is isomorphic to the flag variety $\Levi_\Theta/B_\Theta$ for $\Levi_\Theta$ where $B_\Theta = \Levi_\Theta \cap B$.
More precisely, $P/B$ is the image of the closed embedding $\Levi_\Theta/B_\Theta \hookrightarrow G/B$, which is a $W_\Theta$-equivariant map.
Hence, $\iota^*$ is a $W_\Theta$-equivariant surjective map. 
Let $\varpi_1,\ldots,\varpi_n$ be the fundamental weights.
Then one can see that $\varpi_i \in (\RR^{W_\Theta}_+)$ whenever $\alpha_i \notin \Theta$ since $s_j(\varpi_i) = \varpi_i - \delta_{ij} \alpha_j$ where $\delta_{ij}$ is the Kronecker delta. 
Under the Borel's presentation in \eqref{eq:Borel}, the map $\iota^*$ is written as 
\begin{align} \label{eq:iota}
\RR/(\RR^{W}_+) \twoheadrightarrow \RR/(\RR^{W_\Theta}_+) \cong \RR(\Theta)/({\RR(\Theta)}^{W_\Theta}_+); \ \ \ \bar{\alpha} \mapsto \bar{\alpha}
\end{align}
where $\RR(\Theta)$ is the symmetric algebra of the quotient space $\t(\Theta)^*_{\Q} = \t^*_\Q/\spanned_\Q\{\varpi_i \mid \alpha_i \notin \Theta \}$ and $\t(\Theta)^*_{\Q}$ has a root system isomorphic to $\Phi_\Theta$ (cf. \cite[Appendix~B]{EHNT}). 
In particular, the image of $\Theta$ under the natural projection $\xi: \t^*_\Q \twoheadrightarrow \t(\Theta)^*_{\Q}$ forms simple roots and hence we have a section $\t(\Theta)^*_{\Q} \rightarrow \t^*_\Q$ of $\xi$ by sending $\overline{\alpha_i}$ to $\alpha_i$ for any $\alpha_i \in \Theta$.
This induces a section $\RR(\Theta)/({\RR(\Theta)}^{W_\Theta}_+) \rightarrow \RR/(\RR^{W}_+)$ of the map in \eqref{eq:iota} which is a $W_\Theta$-equivariant map.
Thus, we obtain the following lemma.

\begin{lemma} \label{lemma:section_equivariant}
There is a section $s: H^*(P/B) \rightarrow H^*(G/B)$ of $\iota^*$ such that $s$ is a $W_\Theta$-equivariant map.
\end{lemma}

\subsection{Main theorem}

Our main theorem is as follows.

\begin{theorem} \label{theorem:main}
Let $\xx \in \g$ and $H \subset \g$ be a $\p$-Hessenberg space.
\begin{enumerate}
\item[(1)] There is an isomorphism 
\begin{align} \label{eq:LerayHirsch}
H^*(\Hess(\xx,H)) \cong H^*(P/B) \otimes_\Q H^*(\Hess_\Theta(\xx,H)) 
\end{align}
as $W_\Theta$-modules.
\item[(2)] The homomorphism $\pi_H^*: H^*(\Hess_\Theta(\xx,H)) \rightarrow H^*(\Hess(\xx,H))$ induced from the natural projection $\pi_H: \Hess(\xx,H) \to \Hess_\Theta(\xx,H)$ is injective.
\item[(3)] The image of $\pi_H^*$ coincides with $H^*(\Hess(\xx,H))^{W_\Theta({\rm star})}$, which is the invariants in $H^*(\Hess(\xx,H))$ under the star action of $W_\Theta$. 
In other words, the injective homomorphism $\pi_H^*: H^*(\Hess_\Theta(\xx,H)) \hookrightarrow H^*(\Hess(\xx,H))$ yields the isomorphism of graded $\Q$-algebras
\begin{align} \label{eq:main}
H^*(\Hess_\Theta(\xx,H)) \cong H^*(\Hess(\xx,H))^{W_\Theta({\rm star})}.
\end{align} 
\end{enumerate}
\end{theorem}

\begin{remark}
If $\xx$ is a regular nilpotent element in $\g$, then the injectivity of $\pi_H^*$ and the isomorphism in \eqref{eq:main} was proved in \cite[Theorem 5.4]{Hor24}. 
\end{remark}

\begin{proof}[Proof of Theorem~\ref{theorem:main}]
By Lemma~\ref{lemma:fiber bundle} we obtain the following commutative diagram:
\begin{center}
\begin{tikzcd}
H^*(G/P) \arrow[r, "\pi^*"] \arrow[rightarrow,d, "j_{H, \Theta}^*"'] &[0.5em] H^*(G/B) \arrow[rightarrow,d, "j_H^*"'] \arrow[r, "\iota^*"] &[0.5em] H^*(P/B) \\
H^*(\Hess_\Theta(\xx,H)) \arrow[r, "\pi_H^*"] &[0.5em] H^*(\Hess(\xx,H)) \arrow[r, "\iota_H^*"] & H^*(P/B) \arrow[u,equal] 
\end{tikzcd}
\end{center}
where the horizontal arrows are induced from the fiber bundles and the vertical arrows are induced from the inclusion maps $j_H: \Hess(\xx,H) \hookrightarrow G/B$ and $j_{H, \Theta}: \Hess_\Theta(\xx,H) \hookrightarrow G/P$. 
By Lemma~\ref{lemma:section_equivariant} we can take a section $s: H^*(P/B) \rightarrow H^*(G/B)$ of $\iota^*$ such that $s$ is a $W_\Theta$-equivariant map. 
Define $s_H \coloneqq j_H^* \circ s$. 
Note that $\iota_H^*$ is surjective since $\iota^*$ is surjective.
Then one can verify that $s_H$ is a section of $\iota_H^*$ and $s_H$ is a $W_\Theta$-equivariant map.
Applying the Leray--Hirsch theorem for the fibration $P/B \hookrightarrow \Hess(\xx,H) \rightarrow \Hess_\Theta(\xx,H)$, we obtain the following isomorphism of $\Q$-vector spaces 
\begin{align*}
\varphi: H^*(P/B) \otimes_\Q H^*(\Hess_\Theta(\xx,H)) \cong H^*(\Hess(\xx,H)),
\end{align*}
which sends $\sum \alpha \otimes \beta$ to $\sum s_H(\alpha) \cdot \pi_H^*(\beta)$. 
Since the isomorphism $\varphi$ above is a $W_\Theta$-equivariant map by Lemma~\ref{lemma:piHequivariant}, we proved~(1).
The injectivity of $\pi_H^*$ follows from a similar discussion of \cite[Proposition~5.6 (2)]{Hor24}, proving~(2).
We prove the statement~(3). 
Since $\varphi$ is an isomorphism as $W_\Theta$-modules, we obtain 
\begin{align*}
H^*(\Hess(\xx,H))^{W_\Theta({\rm star})} & \cong \big(H^*(P/B) \otimes_\Q H^*(\Hess_\Theta(\xx,H)) \big)^{W_\Theta({\rm star})} \\
& = H^*(P/B)^{W_\Theta({\rm star})} \otimes_\Q H^*(\Hess_\Theta(\xx,H)) \\ 
& \cong \Q \otimes_\Q H^*(\Hess_\Theta(\xx,H)) \\ 
& \cong H^*(\Hess_\Theta(\xx,H)), 
\end{align*}
where the second equality follows from Lemma~\ref{lemma:piHequivariant}. 
The isomorphism above is exactly given by $\pi_H^*$. 
This completes the proof.
\end{proof}

%%%%%%%%%%%%%%%%%%%%%%%%%%%%%%%%%%
\section{Equivariant cohomology of regular semisimple partial Hessenberg varieties} \label{sect:equivariant cohomology of regular semisimple}
%%%%%%%%%%%%%%%%%%%%%%%%%%%%%%%%%%

\subsection{Regular semisimple Hessenberg varieties}
An element $\ss \in \g$ is \emph{regular semisimple} if its centralizer in $G$ is a maximal torus. 
Let $\ss \in \g$ be a regular semisimple element and $H \subset \g$ a Hessenberg space. 
Then $\Hess(\ss,H)$ is called a \emph{regular semisimple Hessenberg variety}.
It is known that the regular semisimple Hessenberg variety $\Hess(\ss,H)$ is smooth equidimensional of dimension $\dim H/\b$ (\cite[Theorem~6]{dMPS}). 
We may assume that the regular semisimple element $\ss$ is in $\t$ because $\Hess(\ss,H)$ is isomorphic to $\Hess(\ss',H)$ when $\ss$ and $\ss'$ are conjugate in $\g$. Then, the centralizer of $\ss$ in $G$ is the maximal torus $T$, so the natural action of $T$ on $G/B$ preserves $\Hess(\ss,H)$. 

\begin{proposition}[{\cite[Proposition~3, Lemma~7]{dMPS}}] \label{prop:dMPS}
The following hold. 
\begin{enumerate}
\item The fixed point set $\Hess(\ss,H)^T$ of $T$ in $\Hess(\ss,H)$ coincides with $(G/B)^T \cong W =N_G(T)/T$.
\item For any $z \in N_G(T)$, the tangent space of $\Hess(\ss,H)$ at $w=zT$ is isomorphic to 
\begin{align} \label{eq:regular semisimple tangent space}
T_w \Hess(\ss,H) \cong zH/z\b\qquad \text{as $T$-modules}.
\end{align}
Here, $zH$ (resp. $z\b$) is the image of $H$ (resp. $\b$) under the differential map of the left multiplication $G \rightarrow G$ by $z$, that is, $g \mapsto zg$. 
\end{enumerate}
\end{proposition}

It follows from Proposition~\ref{prop:dMPS}(2) and \eqref{eq:pHessenberg space} that 
\begin{equation} \label{eq:regular semisimple tangent space 2}
T_w\Hess(\ss,H)\cong \bigoplus_{\alpha\in I(H)}\C_{-w(\alpha)}\qquad \text{as $T$-modules},
\end{equation}
where $I(H)$ is the subset of $\Phi^+$ defined in \eqref{eq:pHessenberg space}. In particular, when $H=\mathfrak{g}$, we have $\Hess(\ss, H)=G/B$ and $I(H)=\Phi^+$; so \eqref{eq:regular semisimple tangent space 2} reduces to the following well-known fact: 
\begin{equation} \label{eq:flag tangent space}
T_w(G/B)\cong \bigoplus_{\alpha\in \Phi^+}\C_{-w(\alpha)} \qquad \text{as $T$-modules}.
\end{equation}

Let $ET \rightarrow BT \coloneqq ET/T$ be a universal principal $T$-bundle.
Then the $T$-equivariant cohomology $H^*_T(\Hess(\ss,H))$ of the $T$-space $\Hess(\ss,H)$ is defined by 
\[
H^*_T(\Hess(\ss,H))=H^*(ET \times_T \Hess(\ss,H))
\]
where $ET \times_T \Hess(\ss,H)$ denotes the Borel construction of the $T$-space $\Hess(\ss,H)$.
Note that the first projection $ET \times_T \Hess(\ss,H) \rightarrow BT$ yields an $H^*(BT)$-module structure on $H^*_T(\Hess(\ss,H))$. 
We may think of $BT$ as $(\C P^\infty)^n$, so $H^2(BT; \Z)$ is a free abelian group of rank $n$ and $H^*(BT; \Z)$ is a polynomial ring over $\Z$ in $n$ variables of degree $2$. 
We regard the $1$-dimensional $T$-module $\C_{\chi}$ associated to $\chi \in \t^*_\Z$ as a $T$-line bundle $\C_\chi$ over a point. Then, the $T$-equivariant first Chern class $c_1^T(\C_\chi)$ is an element of $H^2_T(\pt; \Z) = H^2(BT; \Z)$. The correspondence $\chi \mapsto c_1^T(\C_\chi)$ yields an isomorphism $\t^*_\Z \cong H^2(BT; \Z)$ and induces the following isomorphism 
\begin{align} \label{eq:identifyHBT}
\RR = \Sym \t^*_{\Q} \cong H^*(BT); \ \ \ \chi \mapsto c_1^T(\C_\chi).
\end{align}

For any two sets $M$ and $N$, we write $\Map(M,N)$ for the set of maps on $M$ taking a value in $N$.
Since the odd degree cohomology groups of $\Hess(\ss,H)$ vanish (\cite[Theorem~8]{dMPS}), the restriction map 
\begin{align*}
H^*_T(\Hess(\ss,H)) \rightarrow H^*_T(\Hess(\ss,H)^T) = \bigoplus_{w \in W} H^*_T(w) = \Map(W,H^*(BT))
\end{align*}
induced from the inclusion $\Hess(\ss,H)^T \subset \Hess(\ss,H)$ is injective. 
Thus, we may regard $H^*_T(\Hess(\ss,H))$ as an $H^*(BT)$-subalgebra of $\Map(W,H^*(BT))$. 
It is shown in \cite[Theorems 1.1 and 2.4]{GHZ} that $G/B$ is a GKM manifold and its GKM graph $\Gamma$ is as follows. 
\begin{enumerate}
\item The vertices of $\Gamma$ are $W$;
\item Two vertices $v, w$ of $\Gamma$ are on a common edge of $\Gamma$ if and only if $v=ws_\alpha$ for some $\alpha\in \Phi$, where we may assume $\alpha\in \Phi^+$ since $s_\alpha=s_{-\alpha}$; 
\item The label on the common edge above is $w(\alpha)$ up to sign. 
\end{enumerate} 
The edge in (2) above corresponds to a 1-dimensional $T$-orbit $\mathcal{O}$ whose closure $\overline{\mathcal{O}}$ is diffeomorphic to $\C P^1$ and $\overline{\mathcal{O}} \setminus \mathcal{O}$ consists of the two fixed points $v$ and $w$. The label $w(\alpha)$ on the edge is the weight of the orbit $\mathcal{O}$, which means that the isotropy subgroup of the 1-dimensional $T$-orbit $\mathcal{O}$ is $\ker(w(\alpha)\colon T\to \C^*)$. Therefore, the edge in (2) above corresponds to the weight space $\C_{-w(\alpha)}$ in \eqref{eq:flag tangent space}. This together with \eqref{eq:regular semisimple tangent space 2} implies that the GKM graph $\Gamma(H)$ of $\Hess(\ss,H)$ with the $T$-action is a GKM subgraph of $\Gamma$ where the vertex set of $\Gamma(H)$ is the same as that of $\Gamma$, that is $W$, and $v,w\in W$ are on a common edge of $\Gamma(H)$ if and only if $v=ws_\alpha$ for some $\alpha\in I(H)$. This means that the lemma below follows from the GKM theory applied to $\Hess(\ss,H)$ with the $T$-action. 

\begin{lemma} $($\cite[Proposition~8.2]{AHMMS}$)$ \label{lemm:GKM_regular_semisimple}
Let $I(H)$ be the subset of $\Phi^+$ defined in \eqref{eq:pHessenberg space}. Then the image of the restriction map $H^*_T(\Hess(\ss,H)) \hookrightarrow \Map(W,H^*(BT))$ is given by 
\begin{equation} \label{eq:GKMHess(s,H)}
\left\{ f \in \Map(W,H^*(BT)) \mid 
f(w)-f(v) \in (w(\alpha)) \ {\rm if} \ v=ws_\alpha\ {\rm for \ some} \ \alpha \in I(H) \right\}
\end{equation}
where $w(\alpha)$ is regarded as an element of $H^2(BT)$ through the isomorphism \eqref{eq:identifyHBT} and $(w(\alpha))$ denotes the ideal in $H^*(BT)$ generated by $w(\alpha)$. 
\end{lemma}

\subsection{Regular semisimple partial Hessenberg varieties}

One can easily see that the results for regular semisimple Hessenberg varieties in the previous subsection are generalized to regular semisimple partial Hessenberg varieties. In this subsection we briefly explain them. 

\begin{lemma} \label{lemma:partial regular semisimple property}
Let $\ss \in \g$ be a regular semisimple element and $H \subset \g$ a $\p$-Hessenberg space. 
Then the following hold for the regular semisimple partial Hessenberg variety $\Hess_\Theta(\ss,H)$. 
\begin{enumerate}
\item[(1)] $\Hess_\Theta(\ss,H)$ is smooth equidimensional of dimension $\dim H/\p$; 
\item[(2)] The odd degree cohomology groups of $\Hess_\Theta(\ss,H)$ vanish. 
\end{enumerate}
\end{lemma}

\begin{proof}
Consider the fiber bundle $P/B \xhookrightarrow{\iota_H} \Hess(\ss,H) \xrightarrow{\pi_H} \Hess_\Theta(\ss,H)$ in Lemma~\ref{lemma:fiber bundle}. Since both $P/B$ and $\Hess(\ss,H)$ are smooth of dimensions $\dim \p/\b$ and $\dim H/\b$ respectively, statement (1) follows. Moreover, statement (2) follows from \eqref{eq:LerayHirsch}. 
\end{proof}

For a regular semisimple element $\ss$ in $\t$ and a $\p$-Hessenberg space $H \subset \g$, the natural $T$-action on $G/P$ preserves $\Hess_\Theta(\ss,H)$ and the $T$-fixed point set $\Hess_\Theta(\ss,H)^T$ coincides with $(G/P)^T \cong W/W_\Theta$. 
In what follows, we write $\overline{w} \in W/W_\Theta$ for the left coset of $W_\Theta$ with representative $w \in W$. 

\begin{lemma} \label{lemm:partial_regular semisimple tangent space}
Let $\ss$ be a regular semisimple element in $\t$.
For $z \in N_G(T)$, we set $w=zT \in W$. Then 
\begin{align} \label{eq:partial regular semisimple tangent space}
T_{\overline{w}} \Hess_\Theta(\ss,H) \cong zH/z\p\qquad\text{as $T$-modules}.
\end{align}
\end{lemma}

\begin{proof}
The fiber bundle $P/B \xhookrightarrow{\iota_H} \Hess(\ss,H) \xrightarrow{\pi_H} \Hess_\Theta(\ss,H)$ induces by differentiation an exact sequence $0 \rightarrow \p/\b \rightarrow T_{e} \Hess(\ss,H) \rightarrow T_{\overline{e}} \Hess_\Theta(\ss,H) \rightarrow 0$ as $T$-modules.
This with \eqref{eq:regular semisimple tangent space} yields an isomorphism 
\begin{align} \label{eq:partial regular semisimple tangent space identity}
T_{\overline{e}} \Hess_\Theta(\ss,H) \cong H/\p\qquad\text{as $T$-modules}.
\end{align}
Here we note that \eqref{eq:partial regular semisimple tangent space identity} holds for an arbitrary regular semisimple element $\ss$ in $\t$. We set $\ss'= \Ad(z^{-1})(\ss)$. Then the left multiplication by $z$ gives an isomorphism $\Hess_\Theta(\ss',H) \cong \Hess_\Theta(\ss,H)$. This induces an isomorphism $T_{\overline{e}} \Hess_\Theta(\ss',H) \cong T_{\overline{w}} \Hess_\Theta(\ss,H)$. Then the $T$-action on $T_{\overline{w}} \Hess_\Theta(\ss,H)$ becomes, under the isomorphism \eqref{eq:partial regular semisimple tangent space identity}, the action of $T$ on $H/\p$ given by $t \cdot [\vv] = z^{-1}tz[\vv]$ for $t \in T$ and $[\vv] \in H/\p$ (twisted action by $z^{-1}$). In other words, $T_{\overline{w}} \Hess_\Theta(\ss,H)$ is isomorphic to $zH/z\p$ as $T$-modules.
\end{proof}

It follows from \eqref{eq:pHessenberg space} and \eqref{eq:partial regular semisimple tangent space} that 
\begin{align} \label{eq:partial regular semisimple tangent space 2}
T_{\overline{w}} \Hess_\Theta(\ss,H) \cong \bigoplus_{\alpha \in I(H) \setminus \Phi^+_\Theta} \C_{-w(\alpha)}\qquad\text{as $T$-modules}.
\end{align}
Since the odd degree cohomology groups of $\Hess_\Theta(\ss,H)$ vanish by Lemma~\ref{lemma:partial regular semisimple property}, the restriction map 
\begin{align*}
H^*_T(\Hess_\Theta(\ss,H)) \rightarrow H^*_T(\Hess_\Theta(\ss,H)^T) = \bigoplus_{\overline{w} \in W/W_\Theta} H^*_T(\overline{w}) = \Map(W/W_\Theta,H^*(BT))
\end{align*}
is injective. We may regard $H^*_T(\Hess_\Theta(\ss,H))$ as an $H^*(BT)$-subalgebra of $\Map(W/W_\Theta,H^*(BT))$. Similarly to Lemma~\ref{lemm:GKM_regular_semisimple}, the following lemma follows from \cite[Theorems 1.1 and 2.4]{GHZ} for $G/P$ and \eqref{eq:partial regular semisimple tangent space 2} by applying the GKM theory to $\Hess_\Theta(\ss,H)$. 

\begin{lemma} \label{lemma:partial regular semisimple GKM}
Let $\ss \in \g$ be a regular semisimple element in $\t$ and $H \subset \g$ a $\p$-Hessenberg space. 
Then the image of the restriction map $H^*_T(\Hess_\Theta(\ss,H)) \hookrightarrow \Map(W/W_\Theta,H^*(BT))$ is given by 
\begin{equation} \label{eq:GKMHessTheta(s,H)}
\left\{ f \in \Map(W/W_\Theta,H^*(BT)) \mid
f(\overline{w})-f(\overline{v}) \in (w(\alpha)) \ {\rm if} \ \overline{v}=\overline{ws_\alpha} \ {\rm for \ some}\ \alpha \in I(H) \setminus \Phi^+_\Theta
\right\}
\end{equation}
where $w(\alpha)$ is regarded as an element of $H^2(BT)$ through the isomorphism \eqref{eq:identifyHBT} and $(w(\alpha))$ denotes the ideal in $H^*(BT)$ generated by $w(\alpha)$ as before. 
\end{lemma}

\subsection{Star actions on equivariant cohomology}

Recall in Section~\ref{subsection:star action} that if $H$ is a $\p$-Hessenberg space, then a regular semisimple Hessenberg variety $\Hess(\ss,H)$ admits a $W_\Theta$-action. 
By definition this $W_\Theta$-action commutes with the $T$-action, so we obtain a $W_\Theta$-action on the $T$-equivariant cohomology ring $H^*_T(\Hess(\ss,H))$. 
We also call this action the \emph{star action} of $W_\Theta$ on $H^*_T(\Hess(\ss,H))$, which is denoted by $w \ast f \in H^*_T(\Hess(\ss,H))$ for $w \in W_\Theta$ and $f \in H^*_T(\Hess(\ss,H))$.
Under the identification of $H^*_T(\Hess(\ss,H))$ with the set in \eqref{eq:GKMHess(s,H)}, the star action is given as follows.

\begin{lemma} \label{lemma:star action formula}
Let $\ss \in \g$ be a regular semisimple element in $\t$ and $H \subset \g$ a $\p$-Hessenberg space. 
For $w \in W_\Theta$ and $f \in H^*_T(\Hess(\ss,H))$, we obtain
\begin{align} \label{eq:star action}
(w \ast f)(v) = f(vw) 
\end{align}
for all $v \in W$.
\end{lemma}

\begin{proof}
Recall that we identify $\Hess(\ss,H)^T$ with the Weyl group $W$. 
The right action of $w \in W_\Theta$ on $ET \times_T \Hess(\ss,H)$ sends $[\ee,gB]$ to $[\ee,gB \cdot w]$ where $gB \cdot w$ is defined in \eqref{eq:right action}. 
By restricting the action onto the $T$-fixed points, we obtain the map $ET \times_T \Hess(\ss,H)^T \rightarrow ET \times_T \Hess(\ss,H)^T$, which sends $[\ee,v]$ to $[\ee,vw]$.
This induces the following commutative diagram
\begin{center}
\begin{tikzcd}
H^*_T(\Hess(\ss,H)) \arrow[rightarrow,d, "w\ast"'] \arrow[hookrightarrow,r, ""] &[0.5em] H^*_T(\Hess(\ss,H)^T) \arrow[rightarrow,d, "w\ast"] \\
H^*_T(\Hess(\ss,H)) \arrow[hookrightarrow,r, ""] & H^*_T(\Hess(\ss,H)^T). 
\end{tikzcd}
\end{center}
Noting that the rightmost map sends $(f(v))_{v \in W}$ to $(f(vw))_{v \in W}$, we obtain the desired result.
\end{proof}

\begin{remark} \label{remark:sta action trivial BT}
The first projection $ET \times_T \Hess(\ss,H) \rightarrow BT$ is a $W_\Theta$-equivariant map where $W_\Theta$ acts on $BT$ trivially. 
In particular, the star action of $W_\Theta$ on $H^*_T(\Hess(\ss,H))$ is trivial on $H^*(BT)$.
\end{remark}

\begin{proposition} \label{proposition:equivariant}
Let $\ss \in \g$ be a regular semisimple element in $\t$ and $H \subset \g$ a $\p$-Hessenberg space. 
\begin{enumerate}
\item[(1)] The homomorphism $(\pi_H^*)_T: H^*_T(\Hess_\Theta(\ss,H)) \rightarrow H^*_T(\Hess(\ss,H))$ induced from the natural projection $\pi_H: \Hess(\ss,H) \to \Hess_\Theta(\ss,H)$ is injective.
\item[(2)] The image of $(\pi_H^*)_T$ coincides with $H^*_T(\Hess(\ss,H))^{W_\Theta({\rm star})}$, which is the invariants in $H^*_T(\Hess(\ss,H))$ under the star action of $W_\Theta$. 
In other words, the injective homomorphism $(\pi_H^*)_T: H^*_T(\Hess_\Theta(\ss,H)) \hookrightarrow H^*_T(\Hess(\ss,H))$ yields the isomorphism of graded $H^*(BT)$-algebras
\begin{align*}
H^*_T(\Hess_\Theta(\ss,H)) \cong H^*_T(\Hess(\ss,H))^{W_\Theta({\rm star})}.
\end{align*} 
\end{enumerate}
\end{proposition}

\begin{proof}
Consider the following commutative diagram 
\begin{center}
\begin{tikzcd}
H^*_T(\Hess_\Theta(\ss,H)) \arrow[rightarrow,d, "(\pi_H^*)_T"'] \arrow[hookrightarrow,r, ""] &[0.5em] H^*_T(\Hess_\Theta(\ss,H)^T) = \Map(W/W_\Theta,H^*(BT)) \arrow[rightarrow,d, "p"] \\
H^*_T(\Hess(\ss,H)) \arrow[hookrightarrow,r, ""] & H^*_T(\Hess(\ss,H)^T) = \Map(W,H^*(BT)) 
\end{tikzcd}
\end{center}
where the horizontal arrows are the restriction maps and the rightmost map $p$ is defined by $p(f)(w) = f(\overline{w})$ for $f \in \Map(W/W_\Theta,H^*(BT))$. 
Since $p$ is injective, $(\pi_H^*)_T$ is also injective.
We also see from \eqref{eq:GKMHess(s,H)}, \eqref{eq:GKMHessTheta(s,H)}, and Lemma~\ref{lemma:star action formula} that the image of $(\pi_H^*)_T$ coincides with $H^*_T(\Hess(\ss,H))^{W_\Theta({\rm star})}$, as desired.
\end{proof}

\subsection{Dot actions on equivariant cohomology}

Tymoczko introduced in \cite{Tym08} a $W$-action on the $T$-equivariant cohomology of regular semisimple Hessenberg varieties by using the GKM description.
This is called the dot action. 
In this section we briefly explain Tymoczko's dot action.

We first define the \emph{dot action} of $W$ on $\Map(W,H^*(BT))$ by the following formula:
\begin{align} \label{eq:dot action}
(w \cdot f)(v) = w(f(w^{-1}v)) \ \ \ {\rm for} \ w, v \in W \ {\rm and} \ f \in \Map(W,H^*(BT)). 
\end{align}
Here, we note that $W$ naturally acts on $H^*(BT)$ by the identification in \eqref{eq:identifyHBT}, namely $w(c_1^T(\C_\alpha)) = c_1^T(\C_{w(\alpha)})$. 
The star action on $H^*(BT)$ is trivial (see Remark~\ref{remark:sta action trivial BT}), while the dot action on $H^*(BT)$ is non-trivial.
Let $\ss \in \g$ be a regular semisimple element in $\t$ and $H \subset \g$ a Hessenberg space. 
Recall that the $T$-equivariant cohomology ring of the regular semisimple Hessenberg variety $\Hess(\ss,H)$ is given by \eqref{eq:GKMHess(s,H)}.
Then the dot action of $W$ defined in \eqref{eq:dot action} preserves the set in \eqref{eq:GKMHess(s,H)} (cf. \cite[Lemma~8.7]{AHMMS}). 
In other words, the dot action of $W$ is defined on the $T$-equivariant cohomology ring $H^*_T(\Hess(\ss,H))$. 
Since the ordinary cohomology of $\Hess(\ss,H)$ is isomorphic to the quotient ring $H^*_T(\Hess(\ss,H))/(H^{>0}(BT))$ where $(H^{>0}(BT))$ is the ideal of $H^*_T(\Hess(\ss,H))$ generated by the positive degree part of $H^*(BT)$, the dot action of $W$ on $H^*_T(\Hess(\ss,H))$ induces a $W$-action on $H^*(\Hess(\ss,H))$. 
This is also called the dot action of $W$ on $H^*(\Hess(\ss,H))$. 

\begin{lemma} \label{lemma:stardotcommute}
Let $\ss \in \g$ be a regular semisimple element in $\t$ and $H \subset \g$ a $\p$-Hessenberg space.
Then, the star action of $W_\Theta$ commutes with the dot action of $W$ on $H^*_T(\Hess(\ss,H))$ (and $H^*(\Hess(\ss,H))$).
\end{lemma}

It follows from Lemma~\ref{lemma:stardotcommute} that the star action of $W_\Theta$ on $H^*(\Hess(\ss,H))$ preserves the invariants $H^*(\Hess(\ss,H))^{W({\rm dot})}$ under the dot action of $W$. 
Conversely, the dot action of $W$ on $H^*(\Hess(\ss,H))$ preserves the invariants $H^*(\Hess(\ss,H))^{W_\Theta({\rm star})}$ under the star action of $W_\Theta$. 

\begin{proof}[Proof of Lemma~\ref{lemma:stardotcommute}]
Let $w \in W, u \in W_\Theta$ and $f \in H^*_T(\Hess(\ss,H))$.
By using \eqref{eq:star action} and \eqref{eq:dot action}, we have 
\begin{align*}
&(u \ast (w \cdot f))(v) = (w \cdot f)(vu) = w(f(w^{-1}vu)) \ {\rm and} \\
&(w \cdot (u \ast f))(v) = w((u \ast f)(w^{-1}v)) = w(f(w^{-1}vu)) \ {\rm for \ all} \ v \in W. 
\end{align*}
Hence, the equality $(u \ast (w \cdot f)) = (w \cdot (u \ast f))$ holds.
\end{proof}

Let $H \subset \g$ be a $\p$-Hessenberg space.
One can define the dot action of $W$ on the (equivariant and ordinary) cohomology of a regular semisimple partial Hessenberg variety $\Hess_\Theta(\ss,H)$ by a similar discussion.
In fact, we define the \emph{dot action} of $W$ on $\Map(W/W_\Theta,H^*(BT))$ by 
\begin{align*} 
(w \cdot f)(\overline{v}) = w(f(\, \overline{w^{-1}v} \, )) \ \ \ {\rm for} \ w \in W, \overline{v} \in W/W_\Theta, \ {\rm and} \ f \in \Map(W/W_\Theta,H^*(BT)). 
\end{align*}
By a similar argument of \cite[Lemma~8.7]{AHMMS}, the dot action above preserves the set \eqref{eq:GKMHessTheta(s,H)}.
Hence we obtain the dot action of $W$ on $H^*_T(\Hess_\Theta(\ss,H))$, which induces that of $W$ on $H^*(\Hess_\Theta(\ss,H))$. 

\begin{proposition} \label{proposition:dotaction starinvariants}
Let $\ss \in \g$ be a regular semisimple element and $H \subset \g$ a $\p$-Hessenberg space.
Then the isomorphism in \eqref{eq:main} for $\xx=\ss$ gives an isomorphism as $W$-modules with respect to the dot action of $W$. 
In particular, we have 
\begin{align} \label{eq:dotaction starinvariants}
H^*(\Hess(\ss,H)) \cong H^*(\Hess(\ss,H))^{W_\Theta({\rm star})} \otimes H^*(P/B) \ \ \ \textrm{as $W$-modules},
\end{align} 
where $H^*(\Hess(\ss,H))^{W_\Theta({\rm star})}$ denotes the invariants under the star action of $W_\Theta$ and the action of $W$ on $H^*(P/B)$ is trivial. 
\end{proposition}

\begin{proof}
As seen in the proof of Theorem~\ref{theorem:main}, one can construct an isomorphism of $W_\Theta$-modules with respect to the star action of $W_\Theta$ by
\begin{align} \label{eq:rhoregsemi}
\varphi: H^*(P/B) \otimes_\Q H^*(\Hess_\Theta(\ss,H)) \cong H^*(\Hess(\ss,H)). 
\end{align}
In what follows, we use the notations appeared in the proof of Theorem~\ref{theorem:main}.
Recall that $\varphi$ is defined by $\varphi(\sum \alpha \otimes \beta) = \sum s_H(\alpha) \cdot \pi_H^*(\beta)$. 

In this proof, the $W$-action means the dot action. 
We now show that the map $\varphi$ also gives an isomorphism of $W$-modules. 
By the definition of dot actions, the homomorphism $(\pi_H^*)_T: H^*_T(\Hess_\Theta(\ss,H)) \rightarrow H^*_T(\Hess(\ss,H))$ is a $W$-equivariant map. 
Thus, $\pi_H^*: H^*(\Hess_\Theta(\ss,H)) \rightarrow H^*(\Hess(\ss,H))$ is also a $W$-equivariant map.
By the definition of dot actions, the homomorphism $j_H^*: H^*(G/B) \rightarrow H^*(\Hess(\ss,H))$ is a $W$-equivariant map.
One can also see that the dot action of $W$ on $H^*(G/B)$ is trivial (cf. \cite[Remark~8.9]{AHMMS}), so $s_H$ is a $W$-equivariant map.
Therefore, we conclude that $\varphi$ also gives an isomorphism of $W$-modules. 

It follows from Lemma~\ref{lemma:stardotcommute} that the dot action of $W$ on $H^*(\Hess(\ss,H))$ preserves the invariants $H^*(\Hess(\ss,H))^{W_\Theta({\rm star})}$. 
By taking the invariant part of both sides in \eqref{eq:rhoregsemi} under the star action of $W_\Theta$, the homomorphism $\pi_H^*$ yields an isomorphism $H^*(\Hess_\Theta(\ss,H)) \cong H^*(\Hess(\ss,H))^{W_\Theta({\rm star})}$, as discussed in the proof of Theorem~\ref{theorem:main}. 
Since $\varphi$ is an isomorphism of $W$-modules, the isomorphism $H^*(\Hess_\Theta(\ss,H)) \cong H^*(\Hess(\ss,H))^{W_\Theta({\rm star})}$ also gives an isomorphism of $W$-modules, as desired. 
Applying this isomorphism for the isomorphism \eqref{eq:rhoregsemi} yields the isomorphism \eqref{eq:dotaction starinvariants} of $W$-modules.
This completes the proof.
\end{proof}

\begin{remark} \label{remark:dot action trivial}
The dot action of $W$ on $H^*(G/P)$ is trivial because $\pi^*: H^*(G/P) \hookrightarrow H^*(G/B)$ is a $W$-equivariant map with respect to the dot action and the dot action of $W$ on $H^*(G/B)$ is trivial, as seen in the proof of Proposition~\ref{proposition:dotaction starinvariants}.
\end{remark}

%%%%%%%%%%%%%%%%%%%%%%%%%%%%%%%%%%
\section{Cohomology rings of regular partial Hessenberg varieties} \label{sect:cohomology of regular}
%%%%%%%%%%%%%%%%%%%%%%%%%%%%%%%%%%

An element $\yy \in \g$ is \emph{regular} if the centralizer of $\yy$ in $G$ has dimension $n$. 
Recall that $E_\alpha$ is a fixed basis of the root space $\g_\alpha$ for a root $\alpha \in \Phi$.
Any regular element $\yy \in \g$ is conjugate to an element of the form $\yy_\Xi = \ss_\Xi + \nn_\Xi$ for some $\Xi \subset \Delta$ where $\ss_\Xi$ is semisimple in $\t$ and $\nn_\Xi$ is nilpotent so that $\nn_\Xi = \sum_{\alpha \in \Xi} E_\alpha$ and the centralizer $Z_\g(\ss_\Xi)$ of $\ss_\Xi$ in $\g$ is $Z_\g(\ss_\Xi) = \g(\Xi) \oplus Z$. 
Here, $\g(\Xi)$ is the complex semisimple Lie algebra corresponding to the root system $\Phi_\Xi$, and $Z \subset \t$ is the center of $Z_\g(\ss_\Xi)$ (cf. \cite[Section~2]{AFZ}). 
We call such an element $\yy \in \g$ a \emph{regular element associated to $\Xi \subset \Delta$}. 
A partial Hessenberg variety $\Hess_\Theta(\yy,H)$ is called \emph{regular} if $\yy$ is regular.

Let $H$ be a Hessenberg space and $I(H) \subset \Phi^+$ is defined in \eqref{eq:pHessenberg space}. If $I(H)$ contains the set of simple roots $\Delta$, then the regular Hessenberg variety $\Hess(\yy,H)$ is irreducible by \cite[Corollary~14]{Pre18}. 
It is also known that the complex dimension of the regular Hessenberg variety $\Hess(\yy,H)$ is equal to $\dim H/\b$ and the odd degree cohomology groups of regular Hessenberg varieties vanish by \cite[Lemma~2 and Corollary~3]{Pre18}.
Hence, the following results follow from a similar argument of Lemma~\ref{lemma:partial regular semisimple property}, so we omit the proof.

\begin{lemma} 
Let $\yy \in \g$ be a regular element and $H \subset \g$ a $\p$-Hessenberg space. 
Then the followings hold.
\begin{enumerate}
\item[(1)] If $I(H) \supset \Delta$, then $\Hess_\Theta(\yy,H)$ is irreducible.
\item[(2)] The complex dimension of the regular partial Hessenberg variety $\Hess_\Theta(\yy,H)$ is equal to $\dim H/\p$. 
\item[(3)] The odd degree cohomology groups of $\Hess_\Theta(\yy,H)$ vanish. 
\end{enumerate} 
\end{lemma}

\begin{remark}
If $\yy$ is a regular nilpotent element $\nn$, then $\Hess_\Theta(\nn,H)$ is irreducible for arbitrary $\p$-Hessenberg space $H$ (not necessarily $I(H) \supset \Delta$) (cf. \cite[Proposition~4.4]{Hor24}). 
\end{remark}

We now briefly explain a work of B$\breve{a}$libanu and Crooks in \cite{BaCr24}.
We also refer the reader to \cite[Section~5]{GolSin}.
For a Hessenberg space $H \subset \g$, we define the vector bundle $G \times_B H$ over the flag variety $G/B$.
Consider a map 
\begin{align} \label{eq:muH}
\mu_H : G \times_B H \rightarrow \g; \ \ \ [g, \vv] \mapsto \Ad(g)(\vv).
\end{align} 
Note that $\mu_H$ denotes the family of Hessenberg varieties since $\mu_H^{-1}(\xx)$ is isomorphic to the Hessenberg variety $\Hess(\xx,H)$ for any $\xx \in \g$. 
For any $\xx \in \g$, there exists a sufficiently small Euclidean ball $D_\xx \subset \g$ centered at $\xx$ such that the inclusion $\Hess(\xx,H) \cong \mu_H^{-1}(\xx) \hookrightarrow \mu_H^{-1}(D_\xx)$ induces an isomorphism 
\begin{align} \label{eq:iso_open_ball}
H^*(\mu_H^{-1}(D_\xx)) \cong H^*(\mu_H^{-1}(\xx)) \cong H^*(\Hess(\xx,H)).
\end{align}
Let $\g^{\reg}$ (resp. $\g^{\regsemi}$) be the set of regular elements (resp. regular semisimple elements) in $\g$. 
Fix a regular element $\yy \in \g^{\reg}$ associated to $\Xi \subset \Delta$. 
Take a regular semisimple element $\ss$ in $D_\yy$. 
Composing the induced map from the inclusion $\Hess(\ss,H) \cong \mu_H^{-1}(\ss) \hookrightarrow \mu_H^{-1}(D_\yy \cap \g^{\regsemi}) \hookrightarrow \mu_H^{-1}(D_\yy)$ with the inverse of \eqref{eq:iso_open_ball} gives rise to an isomorphism of graded $\Q$-algebras (\cite[Proposition~4.7]{BaCr24}):
\begin{align} \label{eq:dotinvariant}
\eta_\yy: H^*(\Hess(\yy,H)) \cong H^*(\Hess(\ss,H))^{W_\Xi({\rm dot})}
\end{align}
where the right hand side denotes the invariants in $H^*(\Hess(\ss,H))$ under the dot action of $W_\Xi$. 

\begin{remark} \label{remark:invariant}
Let $\nn \in \g$ be a regular nilpotent element.
If $\yy=\nn$, then the isomorphism \eqref{eq:dotinvariant} of graded $\Q$-algebras in type $A$ was proved in \cite{AHHM} by using an explicit presentation of $H^*(\Hess(\nn,H))$.
The isomorphism \eqref{eq:dotinvariant} as $\Q$-vector spaces for type $A$ was proved in \cite{BroCho}.
The isomorphism \eqref{eq:dotinvariant} of graded $\Q$-algebras for the case when $\yy=\nn$ (in arbitrary Lie types) was proved in \cite{AHMMS} from a point of view of hyperplane arrangements.
We also note that Peterson variety $\Pet_\Phi$ is the special case of regular nilpotent Hessenberg varieties and the toric variety $X_\Phi$ associated with the fan of Weyl chambers is the special case of regular semisimple Hessenberg varieties.
Before the works in \cite{AHHM, AHMMS, BroCho}, an explicit presentation of the cohomology ring for $\Pet_\Phi$ was given by \cite{FHM} in type $A$ and \cite{HHM} for all Lie types, and the $W$-invariants of the cohomology ring of $X_\Phi$ was studied by \cite{Kly}. 
Comparing their presentations, one can see that $H^*(\Pet_\Phi)$ is isomorphic to $H^*(X_\Phi)^{W({\rm dot})}$.
\end{remark}

Combining the isomorphism \eqref{eq:dotinvariant} with Theorem~\ref{theorem:main}, we obtain the following result.

\begin{theorem} \label{theorem:main2} 
Let $\yy \in \g$ be a regular element associated to $\Xi \subset \Delta$ and $\ss \in \g$ a regular semisimple element.
Let $H \subset \g$ be a $\p$-Hessenberg space.
Then there is an isomorphism of graded $\Q$-algebras:
\begin{align*}
H^*(\Hess_\Theta(\yy,H)) \cong H^*(\Hess_\Theta(\ss,H))^{W_\Xi({\rm dot})}
\end{align*}
where the right hand side denotes the invariants in $H^*(\Hess_\Theta(\ss,H))$ under the dot action of $W_\Xi$. 
\end{theorem}

\begin{proof}
We first show that the map $\eta_\yy$ in \eqref{eq:dotinvariant} is an isomorphism of $W_\Theta$-modules with respect to the star action of $W_\Theta$. 
Recall from \eqref{eq:vector bundle homeo} that the homeomorphism $K/T_K \approx G/B$ induces the homeomorphism $K \times_{T_K} H \approx G \times_B H$.
We regard $W_\Theta$ as a subgroup of $N_K(T_K)/T_K (\cong W)$.
Since $H$ is a $\p$-Hessenberg space, one can define a right action of $W_\Theta$ on $K \times_{T_K} H$ by $[k,\vv] \cdot w \coloneqq [kz,\Ad(z^{-1})(\vv)]$ for $[k,\vv] \in K \times_{T_K} H$ and $w=zT_k \in W_\Theta$. 
Thus, we obtain a right action of $W_\Theta$ on $G \times_B H$ via the homeomorphism $K \times_{T_K} H \approx G \times_B H$. 
For any subset $\a \subset \g$, the right action of $W_\Theta$ on $G \times_B H$ preserves $\mu_H^{-1}(\a)$ where $\mu_H$ is defined in \eqref{eq:muH}, so both inclusions 
\begin{align*}
&\Hess(\xx,H) \cong \mu_H^{-1}(\xx) \hookrightarrow \mu_H^{-1}(D_\xx) \ \textrm{and} \\
&\Hess(\ss,H) \cong \mu_H^{-1}(\ss) \hookrightarrow \mu_H^{-1}(D_\yy \cap \g^{\regsemi}) \hookrightarrow \mu_H^{-1}(D_\yy)
\end{align*}
are $W_\Theta$-equivariant maps. 
By the construction of the map $\eta_\yy$ in \eqref{eq:dotinvariant} and Lemma~\ref{lemma:stardotcommute}, we see that $\eta_\yy$ is an isomorphism of $W_\Theta$-modules with respect to the star action of $W_\Theta$. 
Hence, we have the following isomorphism of graded $\Q$-algebras
\begin{align} \label{eq:regularproof1}
H^*(\Hess_\Theta(\yy,H)) \cong H^*(\Hess(\yy,H))^{W_\Theta({\rm star})} \cong \left( H^*(\Hess(\ss,H))^{W_\Xi({\rm dot})} \right)^{W_\Theta({\rm star})}. 
\end{align}
Here, we used Theorem~\ref{theorem:main} for the first isomorphism. 

On the other hand, it follows from Proposition~\ref{proposition:dotaction starinvariants} that the isomorphism in \eqref{eq:main} for $\xx=\ss$, that is, 
$$
H^*(\Hess_\Theta(\ss,H)) \cong H^*(\Hess(\ss,H))^{W_\Theta({\rm star})}
$$
gives an isomorphism of $W$-modules as the dot action.
Therefore, we obtain the following isomorphism of graded $\Q$-algebras
\begin{align} \label{eq:regularproof2}
H^*(\Hess_\Theta(\ss,H))^{W_\Xi({\rm dot})} \cong \left(H^*(\Hess(\ss,H))^{W_\Theta({\rm star})} \right)^{W_\Xi({\rm dot})}.
\end{align}
By \eqref{eq:regularproof1} and \eqref{eq:regularproof2} we obtain the desired isomorphism.
\end{proof}

Let $\ss \in \g$ be a regular semisimple element and $\Xi \subset \Delta$.
By a similar argument of \cite[Proposition~10.6]{AHHM} and \cite[Proposition~8.13 and Theorem~12.1]{AHMMS}, we can prove that $H^*(\Hess_\Theta(\ss,H))^{W_\Xi({\rm dot})}$ satisfies a Poincar\'e duality algebra, the hard Lefschetz property, and the Hodge--Riemann relations, under the assumption that the top degree of $H^*(\Hess_\Theta(\yy,H))^{W_\Xi({\rm dot})}$ is $1$-dimensional vector space. 
Therefore, we conclude the following result from Theorem~\ref{theorem:main2}.

\begin{corollary} \label{coro:regular}
Let $\yy \in \g$ be a regular element and $H \subset \g$ a $\p$-Hessenberg space. 
Let $d = \dim_\C \Hess_\Theta(\yy,H) = \dim H/\p$.
Assume that the regular partial Hessenberg variety $\Hess_\Theta(\yy,H)$ is irreducible. 
We fix an isomorphism $\int: H^{2d}(\Hess_\Theta(\yy,H)) \cong \Q$.
Then the followings hold.
\begin{enumerate}
\item[(1)] The cohomology ring $H^*(\Hess_\Theta(\yy,H))$ is a Poincar\'e duality algebra defined as follows. 

\bigskip
\noindent
\textbf{Poincar\'e duality algebra:} 
The map
\begin{align*}
H^{2k}(\Hess_\Theta(\yy,H)) \times H^{2(d-k)}(\Hess_\Theta(\yy,H)) \rightarrow \Q; \ (\alpha,\beta) \mapsto \int \alpha \cup \beta 
\end{align*}
is non-degenerate for all $0 \leq k \leq \lfloor \frac{d}{2} \rfloor$. 

\item[(2)] If $d \geq 1$, then there exists a non-zero element $\omega \in H^2(\Hess_\Theta(\yy,H))$ such that $\omega$ satisfies the hard Lefschetz property and the Hodge--Riemann relations as follows.

\bigskip
\noindent
\textbf{Hard Lefschetz property:} The multiplication by $\omega^{d-2k}$ gives an isomorphism 
\begin{align*}
H^{2k}(\Hess_\Theta(\yy,H)) \rightarrow H^{2(d-k)}(\Hess_\Theta(\yy,H))
\end{align*}
for all $0 \leq k \leq \lfloor \frac{d}{2} \rfloor$. 

\bigskip
\noindent
\textbf{Hodge--Riemann relations:} For any $0 \leq k \leq \lfloor \frac{d}{2} \rfloor$, the symmetric bilinear form
\begin{align*}
H^{2k}(\Hess_\Theta(\yy,H)) \times H^{2k}(\Hess_\Theta(\yy,H)) \rightarrow \Q; \ (\alpha,\beta) \mapsto (-1)^k \int \omega^{d-2k} \cup \alpha \cup \beta 
\end{align*}
is positive-definite on the kernel of the multiplication map
\begin{align*}
\omega^{d-2k+1}: H^{2k}(\Hess_\Theta(\yy,H)) \rightarrow H^{2(d-k+1)}(\Hess_\Theta(\yy,H)). 
\end{align*}
\end{enumerate}
\end{corollary}

%%%%%%%%%%%%%%%%%%%%%%%%%%%%%%%%%%
\section{Poincar\'e polynomial} \label{sect:Poincare polynomial}
%%%%%%%%%%%%%%%%%%%%%%%%%%%%%%%%%%

In this section we give an explicit formula for the Poincar\'e polynomial of the regular semisimple Hessenberg variety associated with the double lollipop Hessenberg function in type $A$.

\subsection{Toric varieties in $G/P$}

Let $\t^*_{\R}=\t^*_{\Z} \otimes_{\Z} \R$. 
The weight polytope $\PP_\Phi(x)$ of a point $x \in \t^*_{\R}$ is the convex hull of the orbits of $x$ under the Weyl group action as follows:
\begin{align*}
\PP_\Phi(x) \coloneqq \conv\{w \cdot x \mid w \in W \}.
\end{align*}
Let $\varpi_1,\ldots,\varpi_n$ be the fundamental weights and we consider the cone 
\begin{align*}
C_\Theta =\cone(\varpi_i \mid \alpha_i \notin \Theta)
\end{align*}
for $\Theta \subset \Delta$.
Let $\lambda$ be an element in the relative interior of the cone $C_\Theta$.
Then the weight polytope $\PP_\Phi(\lambda)$ is the moment map image of $G/P$ (\cite[Section~8.3]{GelSer}) where $P$ denotes the parabolic subgroup associated with $\Theta$. 
Note that $\dim \PP_\Phi(\lambda) = \rank \Phi = n$. 
Let $X_\Theta$ be the toric variety associated to the weight polytope $\PP_\Phi(\lambda)$, i.e. the corresponding fan is the normal fan to $\PP_\Phi(\lambda)$. 
(Note that $X_\Delta$ is $X_\Phi$ in Remark~\ref{remark:invariant} by abuse of notation.)
By using the relation between $h$-vectors and $f$-vectors, the Poincar\'e polynomial of $X_\Theta$ equals 
\begin{align} \label{eq:Poincare_polynomial_X_Theta_1}
\Poin(X_\Theta,\sqrt{q})=\sum_{i=0}^n f_i (q-1)^i,
\end{align}
where $f_i$ denotes the number of $i$-dimensional faces of $\PP_\Phi(\lambda)$. 
The face structure of weight polytopes $\PP_\Phi(\lambda)$ is well understood as follows.
Let $\Gamma$ be the Dynkin diagram of $\Phi$.
We set 
\begin{align} \label{eq:STheta}
\SS(\Theta) = \{\Xi \subset \Delta \mid \textrm{no connected component of $\Gamma|_\Xi$ is contained in $\Gamma|_\Theta$} \},
\end{align}
where $\Gamma|_\Xi$ denotes the resriction of the Dynkin diagram $\Gamma$ to $\Xi$ in the sense of the induced graph. 
Here, we take the convention that the empty set is always an element of $\SS(\Theta)$. 

\begin{theorem} $($\cite[Corollary~1.3]{Ren}, cf. \cite[Theorem~2.4.3]{GaoMcD} and \cite[Theorem~6.5]{ACEP}$)$ \label{theorem:f-vectors}
For $\Theta \subset \Delta$, let $\lambda$ be an element in the relative interior of the cone $C_\Theta$, so that the isotropy group of $\lambda$ is $W_\Theta$. 
Let $\PP_\Phi(\lambda)$ be the weight polytope associated to $\lambda$.
Then the following holds.
\begin{enumerate}
\item[(1)] There is a bijection
\begin{align*} 
\SS(\Theta) \longleftrightarrow \{\textrm{$W$-orbits of faces of $\PP_\Phi(\lambda)$} \}, 
\end{align*}
where $\SS(\Theta)$ is defined in \eqref{eq:STheta}. 
Under this bijection, an element $\Xi \in \SS(\Theta)$ corresponds to the $W$-orbit of a $|\Xi|$-dimensional face $F_\Xi$ that is combinatorially equivalent to $\PP_{\Phi_\Xi}(\sum_{\alpha_i \in (\Delta \setminus \Theta) \cap \Xi} \varpi_i)$. 
\item[(2)] 
Let $F_\Xi$ denote the $|\Xi|$-dimensional face of $\PP_\Phi(\lambda)$ defined in the statement~$(1)$. 
Then the stabilizer of $F_\Xi$ is generated by the reflections with respect to roots in $\Xi^*$ defined as follows
\begin{align*}
\Xi^* \coloneqq \Xi \cup \{\alpha_i \in \Theta \mid s_i s_j =s_j s_i \ \textrm{for any} \ \alpha_j \in \Xi \}.
\end{align*}
\end{enumerate}
\end{theorem}

By Theorem~\ref{theorem:f-vectors}, the number of $i$-dimensional faces of $\PP_\Phi(\lambda)$ equals
\begin{align*} 
f_i = \sum_{\Xi \in \SS(\Theta) \atop |\Xi| = i} \frac{|W|}{|W_{\Xi^*}|}.
\end{align*}
This with \eqref{eq:Poincare_polynomial_X_Theta_1} yields that 
\begin{align} \label{eq:Poincare_polynomial_X_Theta_2}
\Poin(X_\Theta,\sqrt{q})= \sum_{\Xi \in \SS(\Theta)} \frac{|W|}{|W_{\Xi^*}|} (q-1)^{|\Xi|}.
\end{align}

\subsection{Regular semisimple Hessenberg varieties for double lollipop case in type $A$}

Recall that $X_\Theta$ is the toric variety associated to the weight polytope $\PP_\Phi(\lambda)$ where we take $\lambda$ as an element in the relative interior of the cone $C_\Theta$. 
If $\Theta$ is the empty set, then the toric variety $X_\Theta$ can be realized as the regular semisimple Hessenberg variety $\Hess(\ss,H)$ in the flag variety $G/B$ where $I(H) = \Delta$ in \eqref{eq:pHessenberg space} (\cite[Theorem~11]{dMPS}). 
In type $A$, this result is generalized to partial flag varieties (\cite[Lemma~4.1]{MasSat}). 
We first explain this result below. 
In what follows, we frequently write 
\begin{align*}
[n] \coloneqq \{1,2, \ldots, n \}
\end{align*}
for any positive integer $n$. 
For $1 \leq a < b \leq n-1$, we also denote the consecutive substring from $a$ to $b$ by 
\begin{align*}
[a,b] \coloneqq \{a,a+1, \ldots, b\}.
\end{align*}

Let $1 \leq a < b \leq n-1$. 
Consider the following partial flag variety in type $A_{n-1}$
\begin{align*}
\Flag_{[a,b]}(\C^n) \coloneqq \{ (V_a \subset V_{a+1} \subset \cdots \subset V_b \subset \C^n) \mid \dim_\C V_i = i \ \textrm{for all} \ a \leq i \leq b \}.
\end{align*}
Let $S$ be a regular semisimple matrix (i.e. $S$ has distinct eigenvalues). 
We define 
\begin{align*}
X_{[a,b]} \coloneqq \{ (V_a \subset V_{a+1} \subset \cdots \subset V_b \subset \C^n) \in \Flag_{[a,b]}(\C^n) \mid SV_i \subset V_{i+1} \ \textrm{for all} \ a \leq i < b \}. 
\end{align*}
By \cite[Lemma~4.1]{MasSat}, $X_{[a,b]}$ is the toric variety associated to the weight polytope $\PP_\Phi(\lambda)$ where $\lambda$ is an element in the relative interior of the cone $C_\Theta$ and $\Theta = \Delta \setminus \{\alpha_i \mid a \leq i \leq b \}$. 
As usual, we number the vertices of the Dynkin diagram of type $A_{n-1}$ by $[n-1] = \{1,2, \ldots, n-1 \}$ and regard the vertices as the simple roots $\alpha_1, \ldots , \alpha_{n-1}$. 
We paint the vertices in $\Theta$ black as shown in Figure~\ref{pic: Dynkin diagram}. 

\begin{figure}[h]
\setlength{\unitlength}{1mm}
\begin{center} 
  \begin{picture}(120,15)(0,0)
  \put(0,10){\circle*{2}}
  \put(1,10){\line(1,0){8}}
  \put(10,10){\circle*{2}}
  \put(11,10){\line(1,0){6}}
  \put(20,8.8){$\cdots$}
  \put(28,10){\line(1,0){6}}
  \put(35,10){\circle*{2}}
  \put(36,10){\line(1,0){8}}
  \put(45,10){\circle{2}}
  \put(46,10){\line(1,0){8}}
  \put(55,10){\circle{2}}
  \put(56,10){\line(1,0){6}}
  \put(65,8.8){$\cdots$}
  \put(73,10){\line(1,0){6}}
  \put(80,10){\circle{2}}
  \put(81,10){\line(1,0){8}}
  \put(90,10){\circle*{2}}
  \put(91,10){\line(1,0){6}}
  \put(100,8.8){$\cdots$}
  \put(108,10){\line(1,0){6}}
  \put(115,10){\circle*{2}}
  \put(-1,3){$1$}
  \put(9,3){$2$}
  \put(30,3){$a-1$}
  \put(44,3){$a$}
  \put(51,3){$a+1$}
  \put(79,3){$b$}
  \put(86,3){$b+1$}
  \put(110,3){$n-1$}
  \end{picture}
\end{center}  
\caption{The Dynkin diagram of type $A_{n-1}$ with vertices $\bullet$ in $\Theta = \Delta \setminus \{\alpha_i \mid a \leq i \leq b \}$.}
\label{pic: Dynkin diagram}
\end{figure}

By setting
\begin{align*}
\SS_{[a,b]} &= \{J \subset [n-1] \mid \textrm{no connected component of $J$ is contained in $[n-1] \setminus [a,b]$} \}; \\
J^* &= J \sqcup \{i \in [n-1] \setminus [a,b] \mid |i-j| \geq 2 \ \textrm{for any} \ j \in J \};
\end{align*}
we obtain from \eqref{eq:Poincare_polynomial_X_Theta_2} that 
\begin{align} \label{eq:Poincare_polynomial_X_ab}
\Poin(X_{[a,b]},\sqrt{q})= \sum_{J \in \SS_{[a,b]}} \frac{n!}{|\mathfrak{S}_{J^*}|} (q-1)^{|J|},
\end{align}
where $\mathfrak{S}_{J^*}$ is a subgroup of the symmetric group $\mathfrak{S}_n$ generated by the adjacent transposition $s_i \ (i \in J^*)$ swapping $i$ and $i+1$.

We now define the regular semisimple Hessenberg variety of double lollipop type.
Recall that the flag variety in type $A_{n-1}$ is defined by
\begin{align*}
\Flag(\C^n) =\{ (V_1 \subset V_2 \subset \cdots \subset V_n = \C^n) \mid \dim_\C V_i = i \ \textrm{for all} \ 1 \leq i \leq n \}.
\end{align*}
A function $h:[n] \rightarrow [n]$ is called a \emph{Hessenberg function} if $h$ is weakly increasing with the condition that $h(i) \geq i$ for any $i \in [n]$.
For a Hessenberg function $h$, the regular semisimple Hessenberg variety in type $A_{n-1}$ is defined by 
\begin{align*}
\Hess(S,h) = \{(V_1 \subset V_2 \subset \cdots \subset V_n = \C^n) \in \Flag(\C^n) \mid SV_i \subset V_{h(i)} \ \textrm{for any} \ i \in [n] \}. 
\end{align*}
A Hessenberg function $h$ is called of \emph{double lollipop type for} $[a,b]$ if $h$ is of the form 
\begin{align*}
h(j) = 
\begin{cases}
a+1 & (1 \leq j \leq a); \\
j+1 & (a < j < b); \\
n & (b \leq j \leq n).
\end{cases}
\end{align*}

\begin{proposition} \label{proposition: Poincare_polynomial_double_lollipop}
Let $1 \leq a < b \leq n-1$ and $h$ the Hessenberg function of double lollipop type for $[a,b]$.
Then, the Poincar\'e polynomial of the regular semisimple Hessenberg variety $\Hess(S,h)$ is given by
\begin{align*} 
\Poin(\Hess(S,h),\sqrt{q})=[a]_q! [n-b]_q! \left( \sum_{J \in \SS_{[a,b]}} \frac{n!}{|\mathfrak{S}_{J^*}|} (q-1)^{|J|} \right),
\end{align*}
where we denote $[n]_q = 1+q+q^2+\cdots+q^{n-1}$ and $[n]_q! = \prod_{i=1}^n [i]_q$. 
\end{proposition}

\begin{proof}
By Lemma~\ref{lemma:fiber bundle}, the natural projection $\Hess(S,h) \rightarrow X_{[a,b]}$ is a fiber bundle with fiber isomorphic to $\Flag(\C^a) \times \Flag(\C^{n-b})$ (see also \cite{KiLe} and \cite{MasSat}).
It then follows from Theorem~\ref{theorem:main} (1) that 
\begin{align*} 
\Poin(\Hess(S,h),\sqrt{q})=\Poin(\Flag(\C^a),\sqrt{q}) \Poin(\Flag(\C^{n-b}),\sqrt{q}) \Poin(X_{[a,b]},\sqrt{q}).
\end{align*}
Since $\Poin(\Flag(\C^n),\sqrt{q}) = [n]_q!$, the result follows from \eqref{eq:Poincare_polynomial_X_ab}.
\end{proof}

\begin{example}
Let $n=5$ and we take $a=2$ and $b=3$.
Then we have 
\begin{align*} 
\SS_{[a,b]} = \{\emptyset, \{2\}, \{3\}, \{1,2\}, \{2,3\}, \{3,4\}, \{1,2,3\}, \{2,3,4\}, \{1,2,3,4\} \}.
\end{align*}
By Proposition~\ref{proposition: Poincare_polynomial_double_lollipop} we obtain 
\begin{align*} 
\Poin(\Hess(S,h),\sqrt{q})= &(1+q)^2 \big(\frac{5!}{2!2!} + 2\frac{5!}{2!2!}(q-1) + 2\frac{5!}{3!2!}(q-1)^2 + \frac{5!}{3!}(q-1)^2 \\
& \hspace{60pt} + 2\frac{5!}{4!}(q-1)^3 + \frac{5!}{5!}(q-1)^4 \big) \\
=&(1+q)^2 \big(30+60(q-1) + 40(q-1)^2 + 10(q-1)^3 + (q-1)^4 \big) \\ 
=&(1+q)^2 \big(q^4+6q^3+16q^2+6q+1 \big). 
\end{align*}
\end{example}

\end{document}